\documentclass[11pt,a4paper,twoside]{article}
\usepackage[T1]{fontenc}
\usepackage[utf8]{inputenc}
\usepackage[english]{babel}
\usepackage{amssymb, amsmath, amsfonts, amsthm}
\usepackage{enumerate}
\usepackage{colonequals}
\usepackage{empheq,fancybox}
\usepackage[small]{titlesec}
\usepackage[noblocks]{authblk}

\usepackage{microtype}
\usepackage{ellipsis}
\usepackage[BCOR=20mm,DIV=11]{typearea}
 \newcommand{\hm}[1]{\leavevmode{\marginpar{\tiny%
 $ \hbox to 0mm{\hspace*{-0.5mm} $ \leftarrow $ \hss}%
 \vcenter{\vrule depth 0.1mm height 0.1mm width \the\marginparwidth}%
 \hbox to
 0mm{\hss $ \rightarrow $ \hspace*{-0.5mm}} $ \\\relax\raggedright #1}}}

\newcommand{\euler}{\mathrm{e}}
\newcommand{\drm}{\mathrm{d}}
\newcommand{\dvol}{\mathrm{dvol}}
\newcommand{\RR}{\mathbb{R}}

\newcommand{\NN}{\mathbb{N}}

\newcommand{\tvert}[1]{{\left\vert\kern-0.25ex\left\vert\kern-0.25ex\left\vert #1 
    \right\vert\kern-0.25ex\right\vert\kern-0.25ex\right\vert}}
\renewcommand{\epsilon}{\varepsilon}

\DeclareMathOperator{\diam}{\mathop{diam}}

\DeclareMathOperator{\Vol}{\mathop{Vol}}

\DeclareMathOperator{\Ker}{\mathop{Ker}}

\newtheorem{theorem}{Theorem}[section]

\newtheorem{proposition}[theorem]{Proposition}
\newtheorem{corollary}[theorem]{Corollary}
\theoremstyle{definition}
\newtheorem{definition}[theorem]{Definition}
\theoremstyle{remark}
\newtheorem{remark}[theorem]{Remark}

\usepackage{color}
	\definecolor{darkred}{rgb}{0.5,0,0}
	\definecolor{darkgreen}{rgb}{0,0.5,0}
	\definecolor{darkblue}{rgb}{0,0,0.5}
\begin{document}
\title{The Kato class on compact manifolds with integral bounds on the negative part of Ricci curvature}
\author{Christian Rose}
\author{Peter Stollmann}
\affil{Technische Universit\"at Chemnitz, Faculty of Mathematics, D - 09107 Chemnitz}
\date{\today}
\maketitle
\section{Introduction}
The starting point of the present note is a result by Elworthy and Rosenberg \cite{ElworthyRosenberg-91} who 
provided a variant of a classical result of Bochner that implies vanishing of the first real cohomology group 
$H^1(M)$ of a given compact Riemannian manifold. We refer to the following section for more details and 
summarize the ideas briefly. \\
Bochner's theorem says that $H^1(M)=\{0\}$ provided the Ricci curvature $R(\cdot,\cdot)$ is nonnegative and 
strictly positive somewhere; this is quite obvious from the Weitzenb\"ock formula
$$ \Delta^1=\nabla^\ast\nabla+R(\cdot,\cdot)\text,$$
where $\Delta^1$ denotes the Hodge Laplacian on 1-forms (note our sign convention!). Clearly, $\Delta^1$ will 
be positive definite under the assumptions in Bochner's theorem and so 
$$\Ker(\Delta^1)=H^1(M)=\{0\}\text.$$
The above mentioned result by Elworthy and Rosenberg deals with a somewhat different situation: suppose that 
$R(\cdot,\cdot)$ is positive mostly but allowed to take on negative values. Then we can still deduce that 
$H^1(M)=\{0\}$, provided wells of negative curvature are under control. Here, the conclusion is achieved by 
using semigroup domination, which allows us to deal with the Schr\"odinger operator
$$\Delta+\rho$$
defined on functions, where $\Delta$ is the Laplace Beltrami operator and 
$\rho(\cdot)=\inf\sigma(R(\cdot,\cdot))$ denote the lowest eigenvalue of the Ricci tensor, viewed as a section 
of endomorphisms of $\Lambda^1(M)$. Again we have to show that $\Delta+\rho>0$ (positive definite) and this 
looks like an easy question from the point of view of Schr\"odinger operators. Clearly, if $\rho\geq \rho_0>0$ 
mostly, a control of $(\rho-\rho_0)_{-}=:W$ will give the desired positivity. However, there is a serious catch 
here: we cannot simply treat $W$ as a perturbation, as both $\rho$ and $\Delta$ depend on the metric. The 
problem that we want to solve here is to give a rather explicit condition on $W$ that allows the above 
conclusion. Such a quantitative statement is not contained in \cite{ElworthyRosenberg-91}. \\
Our method is inspired by Schr\"odinger operator theory. Actually, our main result gives that $L^p$-means 
control the Kato condition uniformly for a whole family of Riemannian manifolds. The vanishing 
of $H^1(M)$ will then a rather easy consequence under explicit conditions on $W$. We note that 
\cite{ElworthyRosenberg-91} had been generalized to \cite{RosenbergYang-94}, who saw that integrability 
conditions are the right thing to look for. Actually, Gallot's paper \cite{Gallot-88} contains a positivity 
result for Schrödinger operators that can be regarded as a generalization of what is found in 
\cite{ElworthyRosenberg-91}. \\
Clearly, there are many ways to establish positivity of Schr\"odinger operators. Our main result, Theorem 
\ref{Katocondition} below, provides much more, namely a criterion for potentials to be in the Kato class.
Our second application uses more of the power of the Kato condition. The latter is quite useful in 
deriving mapping properties of semigroups and we exploit this feature in deriving ultracontractivity of the 
heat kernel of $\Delta^1$ which, in turn gives upper bounds on the dimension of $H^1(M)$. This is summarized 
in \ref{cor-dim} below.

\emph{The authors want to thank the referee for useful comments and Gilles Carron for pointing at the paper 
\cite{Aubry-07}. The second named author dedicates this work to the memory of Joe J. Perez: miss you buddy!}

\section{The setup}
Consider a compact manifold $M$ of dimension $d\geq 3$ and recall that $\rho$ denotes the smallest eigenvalue 
of the Ricci tensor $R(\cdot,\cdot)$. As shown in \cite{HessSchraderUhlenbrock-77, HessSchraderUhlenbrock-81}, 
the semigroup of the Hodge Laplacian is dominated by the semigroup of the Schr\"odinger operator $\Delta+\rho$ 
in the sense that 
\begin{align}\label{eq:domination}
\vert \euler^{-t\Delta^1}\omega\vert\leq \euler^{-t(\Delta+\rho)}\vert\omega\vert
\end{align}
pointwise on $M$.
\begin{remark}
If $\Delta+\rho>0$, then $H^1(M)=\{0\}$.
\end{remark}
This easy argument is shown in \cite{ElworthyRosenberg-91}, p. 474: for a harmonic 1-form $\omega$, the left 
hand side of \eqref{eq:domination} is equal to $\vert\omega\vert$ for all $t\ge 0$ while the right hand side 
tends to zero as $t\to\infty$, which gives $\omega=0$. \\
It is now clear, why we are dealing with positivity of Schr\"odinger operators in the sequel. Indeed, we go for 
more and introduce the Kato condition, following \cite{StollmannVoigt-96}; the case at hand is particularly 
easy since the potentials in our application are continuous and thus bounded. The corresponding concepts carry 
over to a large class of measures. 
\\
Since $\Delta$ generates a Dirichlet form, we can define, for $V\in L^\infty(M)$, $\alpha>0$:
$$c_{\rm{Kato}}(V,\alpha):=\Vert (\Delta+\alpha)^{-1}V\Vert_\infty\text.$$
\begin{proposition}
For $V\in L^\infty(M)$, $V\geq 0$, $\alpha>0$:
$$V\leq c_{\rm{Kato}}(V,\alpha)(\Delta+\alpha)$$
in the sense of quadratic forms.
\end{proposition}
This can be found implicitly in p. 459, (2) in \cite{Simon-82}; it can be seen as a very special case of \cite{StollmannVoigt-96}, Theorem 2.1.
\begin{corollary}\label{positivityKato}
Let $W\in L^\infty(M)$, $W\geq 0$ and assume that $c_{\rm{Kato}}(W,\alpha)<1$ for some $\alpha>0$. 
Then $\Delta+\alpha-W>0$.
\end{corollary}
In the application we have in mind, $\rho$ is mostly positive in the sense that, for some $\rho_0>0$,
$$W:=(\rho-\rho_0)_-$$
is small in an appropriate sense. With the previous remark, we get positivity of $\Delta+\rho$, provided $c_{\rm{Kato}}(W,\rho_0)<1.$ However, this is rather implicit, in particular because $\rho$ and $\Delta$ both depend on the Riemannian structure.
\begin{remark}
\begin{enumerate}[(i)]
\item
The Kato class was introduced in \cite{Kato-72} and popularized in \cite{Aizenman-Simon-82,Simon-82} in particular, originally for the Laplacian on $\RR^n$. A potential $V$ is in the Kato class provided $c_{\rm{Kato}}(V,\alpha)\to 0$ for $\alpha\to\infty$, where the original condition is phrased in terms of truncated Greens kernels. For the equivalence, see \cite{Voigt-86}.
\item
In \cite{StollmannVoigt-96}, the Kato class has been extended to measures in a framework, where the Laplacian is generalized to the generator $H$ of a regular Dirichlet form. Here the extended Kato class consists of those functions (resp. measures) for which $c_{\rm{Kato}}(V,\alpha)<1$ for $\alpha$ large enough. It is shown that many important properties carry over from $H$ to $H-V$, especially mapping properties of the semigroup.
\item
Close to our results below is the discussion in \cite{Kuwae-Takahashi-07}, where a general setting is considered.
\item 
We also point out that \cite{Batu-14} contains $L^p$-conditions for the Kato class that are similar to ours, with the main difference that a different class of manifolds is considered.
\end{enumerate}
\end{remark}
\section{Gallot's isoperimetric inequality and heat kernel estimates}
In \cite{Gallot-88}, Gallot proves an isoperimetric inequality in terms of the control of certain $L^p$-means of the negative part of Ricci curvature and derives heat kernel estimates. The latter are fundamental for our treatment here, so we recall them for the convenience of the reader.
\begin{theorem}\label{Gallot}
Let $D>0$ and $\delta>d\geq 3$. For any compact Riemannian manifold $(M,g)$ with $\dim M=d$, $\diam M\leq D$ and Ricci curvature satisfying, for some $\lambda>0$
\begin{align}\label{curvaturecondition}
\frac 1{\Vol(M)}\int \left(\frac{\rho_-}{d-1}-\lambda^2\right)_+^{\frac{\delta}{2}}\dvol\leq \frac 12\left(\frac{\lambda^\delta}{\euler^{\lambda B(\delta,d)D}-1}\right),
\end{align}
where
\[
B(\delta,d)=\left(\frac{2(\delta-1)}{\delta}\right)^{\frac12}(d-1)^{1-1/\delta}\left(\frac{\delta-2}{\delta-d}\right)^{\frac12-\frac{1}{\delta}}\text,
\]
the heat kernel $k_{(M,g)}(t,x,y)=:k(t,x,y)$ can be bounded from above by
\[
k(t,x,y)\leq \frac{1+K'(\delta)\gamma(\lambda,\delta,D,d)^{-\delta/2}}{\Vol(M)}t^{-\delta/2},\quad 0<t\leq 1\text,
\]
where
\[
\gamma(\lambda,\delta,D,d)=B(\delta,d)\lambda\inf\left\{2^{-1/(\delta-1)}, \frac 14\frac 1{\euler^{\lambda B(\delta,d)D}-1}\right\}\text,
\]
and an explicitly computable $K'(\delta)$ depending on $\delta$ only.
\end{theorem}
Note that since $M$ is compact 
condition \eqref{curvaturecondition} will be satisfied for $\lambda$ large enough.
The above theorem is part of Theorem 6, p. 203 in the above mentioned article by Gallot. 
Actually, it is stated that
\[
\Vol (M)k(t,x,y)\leq k^*(\gamma(\lambda,\delta,D,d)t)
\]
where
\[
k^*(t)\leq 1+(C(\delta)t)^{-\delta/2},
\]
and the constant $C(\delta)$ involves Bessel functions and their zeroes, which in turn yields the constant 
$K'(\delta)$.\\
Since uniformity is the big issue here, we will keep the convention to indicate the dependence of constants on 
the relevant parameters.

Roughly speaking, the $d/2+\epsilon$-integrability of $\rho_-$ up to some level $\lambda^2$ 
(recall that $\rho(\cdot)$ denotes the lowest eigenvalue of Ricci curvature $R(\cdot,\cdot)$) controls the heat 
kernel in pretty much a euclidean way; however, an effective dimension $\delta>d$ appears. Note that 
$\gamma(\cdots)$ is decreasing in $\lambda$, as is the left hand side of \eqref{curvaturecondition}, while 
the right hand side of \eqref{curvaturecondition} goes to zero as $\lambda\to0$ as well as $\lambda\to\infty$. 
We use this effect to derive a somewhat weaker statement that does not involve the level parameter $\lambda$ 
anymore. First, let us fix some notions and notation.
\begin{definition}
If $(M,g)$ satisfies the assumptions of the previous theorem we write $M\in\mathcal{M}(\lambda,\delta,D,d)$ 
and set
\[
K(\lambda,\delta,D,d):=K'(\delta)\gamma(\lambda,\delta,D,d)^{-\frac{\delta}{2}}\text.
\]
\end{definition}
To simplify notation further, we write $\tvert{\cdot}_p$ for $L^p$-means, i.~e.
\[
\tvert{f}_p:=\left(\frac 1{\Vol(M)}\int_M\vert f(x)\vert^p\dvol(x)\right)^{\frac{1}{p}}=\Vol(M)^{-\frac{1}{p}}\Vert f\Vert_p\text.
\]
\begin{corollary}\label{weakassumption}
Let $D>0$ and $\delta>d\geq 3$. For any compact Riemannian manifold $(M,g)$ with $\dim M=d$, $\diam M\leq D$ 
and Ricci curvature satisfying 
\[
\tvert{\rho_-}_{\delta/2}\leq (d-1)\left(2(\euler^{\delta-1}-1)\right)^{-\frac{2}{\delta}}\left(\frac{\delta-1}{B(\delta,d)D}\right)^2\text,
\]
where $B(\delta,d)$ is as in Theorem \ref{Gallot}, the heat kernel can be estimated by
\begin{align}\label{heatkernelweak}
k(t,x,y)\leq \frac{1+K(\delta)D^{\frac{\delta}{2}}}{\Vol(M)}t^{-\frac{\delta}{2}},\quad 0<t\leq 1\text.
\end{align}
\end{corollary}
\begin{proof}
We only need a lower bound for the right hand side of \eqref{curvaturecondition} for some fixed $\lambda$. 
The assumption on $\tvert{\rho_-}$
implies \eqref{curvaturecondition} for $$\lambda=\frac{\delta-1}{B(\delta,d)D}\text.$$ 
It remains to plug $\lambda$ into the formula for $\gamma(\lambda,\delta,D,d)\colon$
\begin{align*}
\gamma(\lambda,\delta,D,d)&=B(\delta,d)\lambda\inf\left\{2^{-\frac{1}{\delta-1}}, \frac 14\frac 1{\euler^{\lambda B(\delta,d)D}-1}\right\}\\
&=\frac{\delta-1}{D}\frac 14\frac 1{\euler^{\delta-1}-1},
\end{align*}
so that the assertion follows with
$$K(\delta)=K'(\delta)\left(4 \frac{\euler^{\delta-1}-1}{\delta-1}\right)^{\delta/2}\text.$$
\end{proof}
Comparing to Theorem \ref{Gallot} it is worth mentioned that the curvature condition of Corollary \ref{weakassumption}
is satisfied if one chooses $\delta$ big enough.
Of course, we can use either the more subtle estimate from Gallot's theorem or the above simpler one in all that follows.
\begin{definition}
If $(M,g)$ satisfies the assumptions of the previous corollary, we write $M\in\mathcal{M}(\delta,D,d)$.
\end{definition}
\section{The main result and an application to the vanishing of $H^1(M)$}
For a measurable function $V\geq 0$ on $M$ and $\alpha>0$ we set 
$$c_{\rm{Kato}}(V,\alpha):=\sup_{n\in\NN}\Vert (\Delta+\alpha)^{-1}(V\wedge n)\Vert_\infty\in [0,\infty]\text,$$
where $a\wedge b=\min\{a,b\}$.
(For bounded $V$, we can apply the resolvent, regarded as a bounded operator from $L^\infty$ to $L^\infty$, so the truncation procedure in our definition makes sure that all the norms are defined.) The extended Kato class from \cite{StollmannVoigt-96} consists of those measurable $V\geq 0$ for which $c_{\rm{Kato}}(V,\alpha)<1$ for $\alpha$ large enough, while the Kato class consists of those $V$, for which $c_{\rm{Kato}}(V,\alpha)\to 0$ as $\alpha\to\infty$. This is only slightly stronger than the previous condition. However, as can be seen from \cite{StollmannVoigt-96}, the mapping $\alpha\to c_{\rm{Kato}}(V,\alpha)$ carries useful information. This will be exploited rather heavily in what follows. We write $a\vee b=\max\{a,b\}$.

\begin{theorem}\label{Katocondition}
Let $M\in\mathcal{M}(\delta,D,d)$ and $p>\delta/2$. If $0\leq V\in L^p(M)$, then
\[
c_{\rm{Kato}}(V,\alpha)\leq \left(1+K(\delta)D^{\delta/2}\right)^{1/p}I(\alpha,\delta,p)\tvert{V}_p,
\]
where
$$I(\alpha,\delta,p):=\int_0^\infty\euler^{-\alpha t}(t^{-\delta/2p}\vee 1)\drm t .$$

For $M\in\mathcal{M}(\lambda,\delta,D,d)$ we obtain
\[
c_{\rm{Kato}}(V,\alpha)\leq \left(1+K(\lambda,\delta,D,d)\right)^{1/p}I(\alpha,\delta,p)\tvert{V}_p\text.
\]
\end{theorem}
\begin{remark}\label{LpI}
Since $I(\alpha,\delta,p)\to 0$ as $\alpha\to\infty$ we see that $L^p$-potentials are in the Kato class. More precisely,
\[
\frac 1\alpha\leq I(\alpha,\delta,p)\leq \left(\frac 1\alpha\right)^{1-\delta/2p}\left(\frac{2p}{2p-\delta}+\left(\frac 1\alpha\right)^{\delta/2p}\euler^{-\alpha}\right)\text.
\]
These inequalities follow by integrating from 0 to 1 and 1 to $\infty$ and estimating the exponential factor 
by 1 in the first integral.

\end{remark}
\begin{proof}
By Theorem \ref{Gallot} and Corollary \ref{weakassumption}, respectively, we know that
\[
k(t,x,y)\leq \frac {1+K(\delta)D^{\delta/2}}{\Vol(M)}t^{-\delta/2},\quad 0<t\leq 1
\]
with explicit control on $K(\delta)$. This gives
\[
\Vert \euler^{-t\Delta}\Vert_{1,\infty}\leq \sup_{x\in M}k(t,x,y)\leq \frac {1+K(\delta)D^{\delta/2}}{\Vol(M)}t^{-\delta/2},\quad 0<t\leq 1,
\]
where $\Vert A\Vert_{p,q}$ denotes the norm of $A$ as an operator from $L^p$ to $L^q$. The semigroup property and the fact that $\euler^{-t\Delta}$ acts as a contraction on each $L^p$ implies 
\[
\Vert\euler^{-t\Delta}\Vert_{1,\infty}\leq \frac {1+K(\delta)D^{\delta/2}}{\Vol(M)}(t^{-\delta/2}\vee 1),\quad 0<t<\infty\text.
\]
The Riesz-Thorin convexity theorem gives that
\[
\Vert\euler^{-t\Delta}\Vert_{p,\infty}\leq \left(\frac {1+K(\delta)D^{\delta/2}}{\Vol(M)}\right)^{1/p}(t^{-\delta/2p}\vee 1),\quad 0<t<\infty\text.
\]
Consequently, for $V$ bounded,
\begin{align*}
c_{\rm{Kato}}(V,\alpha)&=\Vert (\Delta+\alpha)^{-1}V\Vert_\infty=\Vert\int_0^\infty \euler^{-\alpha t}\euler^{-t\Delta}V\drm t\Vert_\infty\\
&\leq \int_0^\infty \euler^{-\alpha t}\Vert \euler^{-t\Delta}V\Vert_\infty\drm t\leq \int_0^\infty\euler^{-\alpha t}\Vert \euler^{-t\Delta}\Vert_{p,\infty}\Vert V\Vert_p\drm t\\
&\leq \left(1+K(\delta)D^{\delta/2}\right)^{1/p}\frac{\Vert V\Vert_p}{\Vol(M)^{1/p}}I(\alpha,\delta,p)\text.
\end{align*}
as asserted. 
\end{proof}
As a first consequence, we obtain a quantitative version of the vanishing result of Elworthy and Rosenberg mentioned in the introduction:
\begin{corollary}\label{ERcorollary}
Let $3\leq d<\delta<2p$ and assume that $M\in\mathcal{M}(\delta,D,d)$ and
\[
\tvert{(\rho-\rho_0)_-}_p<\left(1+K(\delta)D^{\delta/2}\right)^{-1/p}I(\rho_0,\delta,p)^{-1}\text.
\]
Then
\begin{enumerate}[(i)]
\item
$H^1(M')=\{0\}$ for any finite cover $M'$ of $M$.
\item
If the fundamental group $\pi^1(M)$ is almost solvable, then $\pi^1(M)$ is finite.
\end{enumerate}
\end{corollary}
\begin{proof}
Following the strategy of \cite{ElworthyRosenberg-91} it remains to show that $\Delta+\rho>0$. Since
\[
\Delta+\rho\geq \Delta+\rho_0-(\rho-\rho_0)_-
\]
and the assumption implies that
$$c_{\rm{Kato}}((\rho-\rho_0)_-,\rho_0)<1$$
in view of Theorem \ref{Katocondition}, the statement follows from Corollary \ref{positivityKato}.
\end{proof}
Corollaries \ref{ERcorollary} and \ref{weakassumption} yield a sufficient condition which involves 
$\tvert{(\rho-\rho_0)_-}_p$ only, independent of the parameter $\delta$:
\begin{corollary}
Let $(M,g)$ be a compact Riemannian manifold with $\dim M=d\geq 3$ and $\diam M\leq D$. For $p>d/2$ there is an 
explicit $c(p,d)>0$ with the following property: \\
If, for some $\rho_0>0$,
\[
\tvert{(\rho-\rho_0)_-}_p<\min\left\{c(p,d)D^{-2},\left(1+K(p+d/2)D^{\frac{2p+d}{4}}\right)^{-\frac 1p}I(\rho_0, p+d/2,p)^{-1}\right\}\text,
\]
then the conclusions of Corollary \ref{ERcorollary} hold. 

In particular, for $0<\rho_0\leq 1$ it is sufficient that 
\[
\tvert{(\rho-\rho_0)_-}_p<\min\left\{c(p,d)D^{-2},\left(1+K(p+d/2)D^{\frac{2p+d}{4}}\right)^{-\frac 1p}\frac{2p-d}{6p-d}\rho_0^\frac{2p-d}{4p}\right\}\text.
\]
\end{corollary}
\begin{proof}
We set $\delta=p+d/2$, so that $d<\delta<2p$. Moreover, we let
\[
c(p,d):=(d-1)\left(\frac{p+d/2-1}{B(p+d/2,d)}\right)^2\left(2(\euler^{p+d/2-1}-1)\right)^{-4/(2p+d-2)}
\]
and get 
\[
\tvert{\rho_-}_{\delta/2}\leq\tvert{(\rho-\rho_0)_-}_p\leq c(p,d)D^{-2}
\]
which gives that $M\in\mathcal{M}(p+d/2,D,d)$. Since $\tvert{(\rho-\rho_0)_-}_p$ then satisfies the 
requirements of the preceding corollary, we are done. The statement concerning the case $\rho_0\in (0,1]$
follows from the estimate for $I(\rho_0,p,p+d/2)$ from Remark \ref{LpI}.
\end{proof}
It is worth pointing out that it is also possible to get rid of the assumption that $\pi_1(M)$ has to be almost 
solvable using purely geometric techniques. This was done in \cite{Aubry-07}.\\
As remarked in the introduction, Gallot's paper \cite{Gallot-88} contains a positivity result for 
Schr\"odinger operators that can be used to obtain a generalization of the positivity result in 
\cite{ElworthyRosenberg-91}. Namely, Proposition 13 from \cite{Gallot-88} together with the control of the 
isoperimetric constant would give a result much in the spirit of the preceding corollaries.
\section{$L^p$-$L^q$ smoothing for the Hodge Laplacian and an upper bound on $b_1(M)$}
Here we use more of the power of the Kato condition. We saw in the preceding section that $c_{\rm{Kato}}(V,\alpha)$ can be used to obtain lower bounds for Schr\"odinger operators. For certain explicit quantitative statements, it is easier to use an equivalent reformulation of the Kato condition in terms of the following quantity:\\
For $0\leq V$ measurable on $M$, $\beta>0$, set
\begin{align}\label{alternativeKato}
b_{\rm{Kato}}(V,\beta):=\sup_{n\in\NN}\int_0^\beta\Vert\euler^{-t\Delta}(V\wedge n)\Vert_\infty\drm t\in [0,\infty]\text.
\end{align}
It is well known that $c_{\rm{Kato}}(V,\alpha)$ and $b_{\rm{Kato}}$ are closely related. More precisely, we infer the following inequality from \cite{Batu-14}:
\begin{align}\label{Katorelation}
(1-\euler^{-\alpha \beta})c_{\rm{Kato}}(V,\alpha)\le b_{\rm{Kato}}(V,\beta)\le \euler^{\alpha \beta}c_{\rm{Kato}}(V,\alpha) ,
\end{align}
meaning that the behavior of $b_{\rm{Kato}}(V,\beta)$ for $\beta\to 0$ controls the behavior of 
$c_{\rm{Kato}}(V,\alpha)$ as $\alpha\to\infty$ and vice versa. 
The constant $b_{\rm{Kato}}$ can be controlled the $L^p$-mean of the involved potential in a similar fashion like the constant $c_{\rm{Kato}}$. 
\begin{proposition}\label{bkato}
Let $3\leq d<\delta<2p$, $D>0$ and $M\in\mathcal{M}(\delta,D,d)$. If $V\in L^p(M)$ and $\beta>0$, then
\[
b_{\rm{Kato}}(V,\beta)\leq \left(1+D^{\delta/2}K(\delta)\right)^{1/p}\tvert{V}_pJ(\beta,\delta,p),
\]
where
$$ J(\beta,\delta,p):=\int_0^\beta\left(t^{-\frac{\delta}{2p}}\vee 1\right)\drm t\text.$$
For $M\in\mathcal{M}(\lambda,\delta,D,d)$ we obtain
\[
b_{\rm{Kato}}(V,\beta)\leq \left(1+K(\lambda,\delta,D,d)\right)^{1/p}\tvert{V}_pJ(\beta,\delta,p)\text.
\]
\end{proposition}
The proof is analogous to the proof of Theorem \ref{Katocondition}.
\begin{remark}
Since $J(\beta,\delta,p)\to 0$ as $\beta\to 0$ we see that $L^p$-potentials are in the Kato class. More precisely,
\[
J(\beta,\delta,p)=\begin{cases}
\frac{2p}{2p-\delta}\beta^\frac{2p-\delta}{2p} & \beta\leq 1,\\
\frac{2p}{2p-\delta}+\beta-1 & \beta >1.
\end{cases}
\]
\end{remark}
The following consequence of the Myadera-Voigt 
perturbation theorem easily carries over to arbitrary positivity preserving semigroups on $L^1$. We state it in the context of our set-up for convenience.

\begin{proposition}\label{11norm}
 Let $(M,g)$ be a compact Riemannian manifold and $V\in L^1(M)$ such that, for some $\beta>0$,
 $$
 b:=b_{\rm{Kato}}(V_-,\beta)<1 .
 $$
 Then
 $$
 \| \euler^{-t(\Delta + V)}\|_{1,1}\le C\euler^{\omega t} ,
 $$
 where
 $$
 C=\frac{1}{1-b}, \omega=\frac{1}{\beta}\log \frac{1}{1-b} .
 $$
\end{proposition}
\begin{proof}
 By the Feynman-Kac formula (alternatively, the Trotter-Kato formula and truncation gives a purely analytic
 argument, see \cite{Voigt-86}) we get that 
 $$\vert \euler^{-t(\Delta+V)}f(x)\vert\leq \left(\euler^{-t(\Delta-V_-)}\vert f\vert\right)(x)$$
 which implies that 
 $$\Vert\euler^{-t(\Delta+V)}\Vert_{1,1}\leq \Vert\euler^{-t(\Delta-V_-)}\Vert_{1,1}\text.$$
 Therefore, we can assume that $V_+=0$.
 Apply \cite{Voigt-77}, Thm 1, using that $L=1$ and $\lambda=0$ in the notation of the latter paper.
\end{proof}
\begin{corollary}\label{ultrafunctions}
 Let $3\leq d<\delta<2p$ and assume that $M\in\mathcal{M}(\delta,D,d)$ and $V\in L^1(M)$ such that, for some $\beta>0$,
 $$
 b:=b_{\rm{Kato}}(V_-,\beta)<1 .
 $$
 Then 
 $$
 \| \euler^{-t(\Delta + V)}\|_{p,\infty}\le 
 \left[\frac{1}{1-b}\right]^{\left( 1+\frac{t}{\beta}\right)\left(1-\frac{1}{p}\right)}\left[c(b,\beta,\delta,D,d,\Vol(M))t^{-\frac{\delta}{2}}\right]^{\frac{1}{p}}\quad\!\!\mbox{for }0<t\le 1,
 $$
 where
 $$
 c(b,\beta,\delta,D,d, \Vol(M))=\left[\frac{2}{1-b}\right]^{\left( (1+\frac{1}{\beta})\frac{1+b}{1-b}+\frac{\delta}{2}\right)}\frac{1+K(\delta)D^{\frac{\delta}{2}}}{vol(M)} .
 $$
 \end{corollary}
\begin{remark}
 As can be seen from the following proof, we get a corresponding statement for 
 $M\in\mathcal{M}(\lambda,\delta,D,d)$ where
 $$
 \bar{c}(b,\beta,\lambda,\delta,D,d)=\left[\frac{2}{1-b}\right]^{\left( (1+\frac{1}{\beta})\frac{1+b}{1-b}+\frac{\delta}{2}\right)}\frac{1+K(\lambda,\delta,D,d)}{vol(M)} .
 $$
\end{remark}
\begin{proof}[Proof of Corollary \ref{ultrafunctions}.]
We follow parts of the proof of Thm 5.1 in \cite{StollmannVoigt-96}.

 Again we can assume that $V_+=0$.
 We fix $\beta >0$ and 
 $b=b_{\rm{Kato}}(V_-,\beta)<1$ as in the assumption; pick $\kappa_0>1$ such that $b\kappa_0<1$, for 
 definiteness, let 
 $$
 \kappa_0:= \frac12\left(1+\frac{1}{b}\right)
 $$
 with conjugate exponent 
 $$
 k_0=\frac{\kappa_0}{\kappa_0 - 1}=\frac{1+b}{1-b} .
 $$
 We can decrease $\kappa_0$ slightly to $\kappa$ in such a way that the exponent $k$ conjugate to 
 $\kappa$ is a natural number. Note that we can achieve 
 $$
 \frac{1+b}{1-b}\le k\le \frac{1+b}{1-b}+1 =\frac{2}{1-b} .
 $$
 We now use the preceding Proposition and get that
 $$
 \| \euler^{-t(\Delta + \kappa V)}\|_{1,1}\le C \euler^{\omega t},
 $$
 where
$$
 C=\frac{1}{1-\kappa b}, \omega=\frac{1}{\beta}\log \frac{1}{1-\kappa b} .
 $$ 
 Chasing the constants in the above mentioned proof of Thm 5.1 in \cite{StollmannVoigt-96}, see p. 129, in particular, gives that, for $0<t\le 1$
 $$
  C'_t:=\| \euler^{-t(\Delta + V)}\|_{1,\infty}\le C_{\frac{t}{k}}C^{k-1}\euler^{\omega t(k-1)},
 $$
 where 
 $$
 C_{\frac{t}{k}}=\| \euler^{-\frac{t}{k}\Delta}\|_{1,\infty}\le Kk^{\frac{\delta}{2}}t^{-\frac{\delta}{2}}
 $$
 and $K$ can be chosen as 
 $$
 K\le \frac{1+K(\delta)D^\frac{\delta}{2}}{vol(M)}\mbox{  for  }M\in\mathcal{M}(\delta,D,d)
 $$
 and 
 $$
  K\le \frac{1+K(\lambda,\delta,D,d)}{\Vol(M)}\mbox{  for  }M\in\mathcal{M}(\lambda,\delta,D,d) .
 $$
 We plug in $C,\omega$ as well as the estimates on $k$ and get
 $$
 C'_t\le K\left[\frac{1}{1-\kappa b}\right]^{\frac{1+b}{1-b}}\left[\frac{1}{1-\kappa b}\right]^{\frac{1+b}{1-b}\cdot\frac{t}{\beta}}\left[\frac{2}{1-b}\right]^{\frac{\delta}{2}}
 t^{-\frac{\delta}{2}}
$$
Since $\kappa b \le \kappa_0 b=\frac12 (b+1)$ by our choice above,

 $$
 C'_t\le K \left[\frac{2}{1-b}\right]^{\frac{1+b}{1-b}\left(1+\frac{t}{\beta}\right)+\frac{\delta}{2}}
 t^{-\frac{\delta}{2}}
 $$
 and, since we are interested in $t\le 1$ only, we get the assertion for $p=1$. An appeal to the Riesz-Thorin convexity theorem gives the assertion for arbitrary $1\le p$, where we use that 
 $$
 \| \euler^{-t(\Delta+V)}\|_{\infty,\infty}\le \left[\frac{1}{1-b}\right]^{\left(1+\frac{t}{\beta}\right)}
 $$
 by Proposition \ref{11norm} and duality.
 
\end{proof}

\begin{corollary}
Let $3\leq d<\delta$ and assume that $M\in\mathcal{M}(\delta,D,d)$ and $b:=b_{\rm{Kato}}(\rho_-,\beta)<1$ for some $\beta>0$. Then  
$$
 \| \euler^{-t\Delta^1}\|_{p,\infty}\le \left[\frac{1}{1-b}\right]^{\left( 1+\frac{t}{\beta}\right)\left(1-\frac{1}{p}\right)}\left[c(b,\beta,\delta,D,d,\Vol(M))t^{-\frac{\delta}{2}}\right]^{\frac{1}{p}}\quad\mbox{ for }0<t\le 1,
$$
where $c(b,\beta,\delta,D,d,\Vol(M))$ is as in \ref{ultrafunctions}.
\end{corollary}

This is clear from semigroup domination and Corollary \ref{ultrafunctions}. Using Theorem \ref{Katocondition} 
one can deduce a uniform control of the $L^p-L^q$-smoothing of the semigroup of the Hodge Laplacian. 
Instead of going for these admittedly rather complicated formulae, we derive a bound on the dimension of the 
first Betti number,
$$b_1(M)=\dim(H^1(M)) ,$$
in terms of the Kato condition:

\begin{corollary}\label{cor-dim}
 Let $3\leq d<\delta$ and assume that $M\in\mathcal{M}(\delta,D,d)$ and $b:=b_{\rm{Kato}}(\rho_-,\beta)<1$ for some $\beta>0$. Then
\begin{align}\label{alternativedimension}
 b_1(M)\le d \cdot \left[\frac{2}{1-b}\right]^{\left( (1+\frac{1}{\beta})\frac{1+b}{1-b}+\frac{\delta}{2}\right)}\left(1+K(\delta)D^{\frac{\delta}{2}}\right)
\end{align}
\end{corollary}
\begin{proof}
 First note that $\euler^{-t\Delta^1}$ leaves $H^1(M)$ invariant and so 
 $$
 \dim(H^1(M))\le tr(\euler^{-t\Delta^1})\mbox{  for  }t>0
 $$
 Moreover, by \cite{HessSchraderUhlenbrock-81} we know that
 $$
 tr(\euler^{-t\Delta^1})\le d \cdot tr(\euler^{-t(\Delta+\rho)})\mbox{  for  }t>0 .
 $$
 The latter trace can be calculated as
 $$
 tr(\euler^{-(\Delta+\rho)})=\int_M k(x,x)dvol(x)
 $$
 with a continuous kernel $k$ (note that $\rho$ is continuous) that can be estimated pointwise by 
 $$
 0\le sup_{x,y\in M} k(x,y)\le \| \euler^{-(\Delta+\rho)}\|_{1,\infty},
 $$
 giving:
  $$
 tr(\euler^{-(\Delta+\rho)})\le vol(M)\| \euler^{-(\Delta+\rho)}\|_{1,\infty},
 $$
 and the latter is estimated in Corollary \ref{ultrafunctions} resulting in a cancelling of the volume term.
\end{proof}

Note that the latter estimate can also be seen as an explicit variant of the bound on $b_1(M)$ in Theorem 11 from
\cite{Gallot-88}. There it is shown that there is a function in certain parameters, amongst them the $L^p$-norm of $\rho_-$ for some $p>\frac{d}{2}$, that gives an upper bound for the first Betti number. It is emphasized that $p=\frac{d}{2}$ does not suffice. In our result above we have an explicit function and the condition is  phrased in terms of the Kato condition rather than in terms of $L^p$-mean, which fits with the latter: the Kato constant can be controlled by the $L^p$-norm of $\rho_-$ for any $p>\frac{d}{2}$ but the limiting case $p=\frac{d}{2}$ is not allowed.\\
Of course, Proposition \ref{bkato} can be plugged in here giving the following estimate:
\begin{corollary}\label{katodimensionH1}
Let $3\leq d<\delta<2p$ and assume that $M\in\mathcal{M}(\delta,D,d)$. Let 
$$\bar c:= \frac{2p}{2p-\delta}\left(1+K(\delta)D^{\delta/2}\right)^{1/p}\text.$$
and $\tvert{\rho_-}_p<\bar c^{-1}$. Then 
\begin{align*}
b_1(M)\le d \cdot \left[\frac{2}{1-\bar c\tvert{\rho_-}}\right]^{\left( 2\frac{1+\bar c\tvert{\rho_-}}{1-\bar c\tvert{\rho_-}}+\frac{\delta}{2}\right)}\left(1+K(\delta)D^{\frac{\delta}{2}}\right)\text.
\end{align*}
\end{corollary}
\begin{proof}
Choosing $\beta=1$ we see by Proposition \ref{bkato} that the assumption on $\tvert{\rho_-}_p$ ensures that 
$$b_{\rm{Kato}}(\rho_-,1)\leq \bar c \tvert{\rho_-}_p<1\text.$$
Plugging this into the preceding result gives the assertion.
\end{proof}
%

\end{document}